\documentclass[12pt]{amsart}
\usepackage{epsf} 
\usepackage{amssymb,latexsym, amsmath, amscd, array} 
\swapnumbers 
\numberwithin{equation}{section}
\def\m{\medskip}

\theoremstyle{plain} 
\newtheorem{thm}{Theorem}[section]

\newtheorem{cor}[thm]{Corollary} 
  
\newtheorem{lemma}[thm]{Lemma}

\newtheorem{op}[thm]{Open Question} 
 
\newtheorem{prop}[thm]{Proposition} 

\newcommand\theoref{Theorem~\ref} 
\newcommand\lemref{Lemma~\ref} 

\newcommand\propref{Proposition~\ref} 
\newcommand\corref{Corollary~\ref} 
\newcommand\secref{Section~\ref}

\theoremstyle{definition} 
\newtheorem{defin}[thm]{Definition}

\newtheorem{rem}[thm]{Remark} 
 
\newtheorem{ex}[thm]{Example}

\def\p{{\noindent \it Proof. }}

\def\ga{\alpha}

\def\C{{\mathbb C}} 
\def\P{{\mathbb P}} 
\def\Z{{\mathbb Z}} 
 
\def\R{{\mathbb R}}

\def\rp{{\R \P}} 
\def\cp{{\C \P}} 

\def\crit{\operatorname{Crit}}
\def\critset{\operatorname{critset}}
\def\Crit{\operatorname{Crit}}

\def\cat{\operatorname{cat}}

\def\cl{\operatorname{cl}}

\def\ts{\times}

\def\wt{\widetilde}
\def\pa{\partial}

\long\def\forget#1\forgotten{} %

\begin{document}

\title[Critical Points]{Maps of Degree 1, Lusternik--Schnirelmann Category, and Critical Points}

\author[D. Kundu and Yu.~Rudyak]{Deep Kundu and Yuli B. Rudyak}

\address{Deep Kundu \newline
Department of Mathematics, University of Florida \newline
358 Little Hall, Gainesville, FL 32611-8105, USA}
\email{deepkundu@ufl.edu}

\address{Yuli B. Rudyak \newline
Department of Mathematics, University of Florida \newline
358 Little Hall, Gainesville, FL 32611-8105, USA}
\email{rudyak@ufl.edu}

\maketitle
\begin{abstract}
Let $\crit M$ denote the minimal number of critical points (not necessarily non-degenerate) on a closed smooth manifold $M$. We are interested in the evaluation of $\crit$. It is worth noting that we do not know yet whether $\crit M$ is a homotopy invariant of $M$. This makes the research of $\crit$ a challenging problem.
 
In particular, we pose the following question: given a map $f: M \to N$ of degree 1 of closed manifolds, is it true that $\crit M \geq \crit N$?  We prove that this holds in dimension 3 or less. Some high dimension examples are considered.  Note also that an affirmative answer to the question implies the homotopy invariance of $\crit$; this simple observation is a good motivation for the research.
\end{abstract}

\section{Introduction}
It is well-known that problems of extrema and, in particular, research of critical points of real-valued functions is one of the central problems of mathematics.
In this paper, we discuss critical points of smooth functions on smooth manifolds.  In particular, we are interested in evaluating the minimal number $\crit M$ of critical points on a closed smooth manifold $M$. Unlike general critical points, the theory of non-degenerate critical points is well developed, see~\cite{Mil, Sma1, Sma2,Wal}. So, we are not focused on non-degenerate critical points only.  It is worth noting that we do not know yet whether $\crit M$ is a homotopy invariant of $M$. This makes the research of $\crit M$ a challenging problem.

The second author posed the following question: Let $f: M \to N$ be a map of degree 1 of closed smooth manifolds. Do we always have an inequality $\crit M \geq \Crit N$? It is worth noting that an affirmative answer to this question implies the homotopy invariance of $\crit$. In this paper we prove that the question has a positive answer for $\dim M,N\leq 3$, see \secref{main}.

\m Below the word 'smooth" denotes "$C^{\infty}$". All manifolds are assumed to be smooth and connected. The sign $\cong$ denotes the homotopy of maps or homotopy equivalence of spaces. The sign $\#$ denotes the connected sum of two manifolds.

\section{Preliminaries. Critical points.}

\begin{defin}
Let $f:M\to \R$ be a smooth function on a closed manifold $M$. A point $p\in M$ is called a {\em critical point} of $f$ if $df(p)=0$. Let $\critset f$ be the set of critical points of $f$, and $\crit f$ the cardinality of $\critset f$. Set $\Crit M=\min\{\crit f\}$ where $f$ runs over all smooth functions $f:M \to \R$.
\end{defin}

\begin{defin}Given a local coordinate system $\{x^1, \ldots, x^n\}$ in a neighborhood of a critical point $p\in M^n$ of $f$, the critical point $p$ is called {\em non-degenerate} if the matrix
\[
\left(\frac{\pa^2f}{\pa x^i\pa x^j}(p)\right)
\]
is non-singular. A smooth function on a manifold is a {\em Morse function} if it has no degenerate critical points.
\end{defin}

\begin{rem}
It is worth noting that there is a grave difference between non-degenerate and general critical points. 
For example, any closed surface admits a function with 3 critical points, while an oriented surface of genus $g$ possesses at least $2g$ non-degenerate critical points.
\end{rem}

\begin{rem}The number $\Crit M$ is finite for all closed manifolds $M$. Indeed, every manifold $M$ admits a Morse function, say, $f$~\cite{Mil}. So all critical points of $f$ are non-degenerate and therefore isolated. Now recall that $M$ is closed (and connected by default). 
\end{rem}

\m This claim can be strengthened as follows (Takens~\cite{Tak}).

\begin{thm}\label{tak}
Let $M$ be a closed connected manifold. Then $\Crit M\leq \dim(M)+1$.
\end{thm}

\begin{proof}
See \cite[Propositiom 2.9 and the following text]{Tak}, cf. also~\cite[Proposition 7.26]{CLOT}.
\end{proof}

\begin{op}
Is $\crit$ a homotopy invariant? In other words, let $M,N$ be two closed manifolds. If $M$ is homotopy equivalent to $N$, is it always true that $\crit M =\crit N$?
\end{op}

\begin{ex}
If $M$ is a closed manifold that is homotopy equivalent to the sphere $S^n, n\neq 4$, then $\crit M=2$. \\
The case $n=2$ is classical, see e.g. \cite{Mun2}. \\
The case $n=3$ is a famous theorem of Perelman (every homotopy 3-sphere is diffeomorphic to $S^3$), see e.g. \cite{Mor}\\ 
The case $n>4$ is exposed in \cite[Theorem H]{Sma2}.\\
The case $n=4$ is an open question.
\end{ex}

More examples later, in section \secref{LuS}.

\m There is a refining of \theoref{tak}.

\begin{thm}\label{t:refin}
Let $M$ be a $p$-connected closed manifold $M$ of dimension $n$, $p\geq 1, n\geq 6$. Let $r(M)\subset \Z$ be such that it contains all $i\in \Z$ with the property that $H_i(M;\Z)\neq 0$ or $H^i(M;\Z)\neq 0$. If $r(M)$ is included in the union of $s$ closed intervals of length $p$, then $\crit  M\leq s$,
\end{thm}

\begin{proof} 
See \cite[Theorem 5.1]{Tak}, cf. also~\cite[Proposition 7.28(i)]{CLOT}.
\end{proof}

 \begin{prop}\label{p:deg}
 Let $f \colon M\to N$ be a map of degree $\pm1$ $($of closed connected orientable manifolds$)$. Then the map $f_* \colon H_*(M) \to H_* (N) $ is an epimorphism and 
$f^* \colon H^* (N)  \to H^*(M)$ is a monomorphism. Furthermore, the map $f_* \colon \pi_1(M)\to \pi_1 (N) $ is an epimorphism.
 \end{prop}   

\begin{proof}
 Given a map $g: X \to Y$. We have $g_*(g^*(y)\frown x)=y\frown g_*(x)$ for all $x\in H_*(X),\,y\in H^*(Y)$. Let $[M]$ and $[N]$ be the fundamental classes of $M$ and $N$ respectively and $f: M \to N$ is a map. Now, if $\deg f=1$ then $f_*[M]=[N]$. 
Then for $0\neq y\in H^*(N)$ we have 
\[
f_*(f^*(y)\frown [M])=y\frown f_*[M]=y\frown [N]\neq 0
\]
 because $[N]$ is the fundamental class. Hence $f^*$ is a monomorphism. Furthermore, given $a\in H_*(N)$ we have $a=u\frown [N]$ for some $u\in H^*(N)$ by Poincare duality. So, 
\[
a=u\frown [N]=u\frown f_*[M]=f_*(f^*(u)\frown [M])
\]
is in the image of $f_*$, and hence $f_*$ is epic.

Concerning $\pi_1$, choose $x\in M$ and suppose $f_*: \pi_1(M, x)\to \pi_1(N, f(x))$ is not epic. Let $p: \wt N \to N$ be the covering map corresponding to the subgroup $f_*(\pi_1(M))$ of $\pi_1(N)$. Then there exists a lifting $\wt f: M\to \wt N$ of $f$ with $p\wt f=f$. Since $\deg p \neq \pm 1$, we conclude that $\deg f \neq \pm 1$. 
\end{proof}

\m So, informally speaking, if $f: M \to N$ is a map of degree 1 then $M$ is more ``massive'' than $N$. 

Thus, it is natural to pose the following question.

\begin{op}[Rudyak]\label{q:main}
Given two manifolds $M,N$, suppose that there exists a map $f: M \to N$ of degree 1. Is it true that $\crit M \geq \crit N$?    
\end{op}

\begin{rem}
The affirmative answer to the question implies the homotopy invariance of $\crit$. Indeed, if there are homotopy equivalences $f: M\to N$ and $g:N \to M$ then $\deg f=1=\deg g$ (for suitable orientations of $M$ and $N$), and thus $\crit M \geq \crit N \geq \crit M$.
\end{rem}

\section{Lusternik--Schnirelmann theory.}\label{LuS}
Next, we define an invariant that can be regarded as a homotopy invariant approximation of $\Crit$.
\begin{defin}
    1. A {\em categorical covering} of a space $X$ is an open covering $\{U_\ga\}$ such that each $U_{\ga}$ is open and contractible in $X$.

2. The {\em Lusternik--Schnirelmann category} (in future LS category) of a space $X$ (denoted by $\cat X$) is the least number $k$ such that there exists a categorical covering $\{U_0, \ldots, U_k\}$ of $X$.\\
If $X$ does not admit any categorical covering, we say that $\cat X$ is not defined.
\end{defin}

\begin{prop}
Lusternik--Schnirelmann category is a homotopy invariant. In another words, if $X$ and $Y$ are homotopy equivalent then $\cat X=\cat Y$
\end{prop}

\m For the proof see e.g. \cite[Theorem 1.25]{CLOT}

\m The next theorem is the main application of LS theory. 

\begin{thm}[Lusternik and Schnirelmann]\label{LS}
Given a smooth function $f: M \to \R$ on a closed smooth manifold $M$, the number of critical points of $f$ is at least $1+\cat M$.
\end{thm}

\m The theorem appeared in late 1920th, see \cite{LS} for greater detail. Historically, Lusternik and Schnirelmann proved the old Poincare conjecture: every Riemannian manifold with the topology of a 2-sphere possessed at least three simple closed geodesics (i.e. three embedded geodesic circles). 
To do this, the authors considered the space of closed geodesics on a Riemannian manifold homeomorphic to $S^2$ and interpreted geodesics as critical points of the length functional and estimated the number of these critical points. From this point of view \theoref{LS}  can be regarded as a preliminary (baby) version of the three closed geodesics theorem. 

\begin{rem}
There are many examples with $\crit M=1+cat M$ and a few of them with $\crit M=\cat M+2$, and currently, we do not know examples with $\crit M >\cat M+2$.

 See~\cite[Open Problem 7.34 and the following text]{CLOT}
\end{rem}

\begin{prop}
If $M$ is a closed $k$-connected manifold then $\cat M\leq \dim M/k$. \label{p:dimcon}
\end{prop}

\begin{proof}
See \cite[Theorem1.50]{CLOT}
\end{proof}

\m A good method for calculating the LS category is the following.

Recall the definition of cup length. Let $R$ be a commutative ring and $X$ be a space.
\begin{defin}
    The cup length of $X$ with coefficients in $R$ is the least integer $k$ (or $\infty$) such that all $(k+1)$-fold cup products vanish in the reduced cohomology $\Tilde{H}^{*}(X;R)$. We denote this by $\cl_R(X)$. 
\end{defin}

\begin{thm}[Froloff-Elsholz]
   We have the following inequality\\
   $\cl_R(X) \leq \cat(X)$.
\end{thm}

For the proof see \cite[Proposition 1.5]{CLOT}.

\begin{ex}\label{e:torus}
1. The cohomology ring of the torus $T^n$ over $\Z$ is an exterior algebra on $n$ generators. So, $\cl_{\Z}(T^n)=n$. Hence, $\cat T^n\geq n$. We also have the inequality $\cat T^n \leq \dim T^n=n$, so we conclude that $\cat T^n=n$.

\m 2. For real projective spaces, we have 
\[
H^*(\rp^n;\Z/{2})=\Z/{2}[u]/(u^{n+1}),\dim u=1.
\]
Hence, $\cl_{\Z/2}(\rp^n)=n$. So, $\cat (\rp^n)\geq n$, and we have $\cat (\rp^n) = n$ because $\cat (\rp^n)\leq \dim (\rp^n) =n$.

\m 3. For  complex projective spaces, we have 
\[
H^*(\cp^n;\Z)=\Z[v]/(v^{n+1}),\dim v=2.
\]
Hence, $\cl_{\Z}(\cp^n)=n$. So, $\cat (\cp^n)\geq n$, and we have $\cat (\cp^n) = n$ because $\cat (\cp^n)\leq \dim (\cp^n)/2 =n$.
\end{ex}

\begin{prop}\label{p:crit=1+dim}
Let $M^n$ be a closed manifold such that $\cat M=\dim M=n$. Then $\crit M=n+1$.
\end{prop}

\p We have $\cat M+1\leq \crit M$ by \theoref{LS} and $\crit M\leq\dim M+1= n+1$ by \theoref{tak}. Thus
\[
n+1=\cat M+1\leq \crit M\leq \dim +1= n+1.\qed
\]

\begin{cor}[Some cases of homotopy invariance of $\crit $]\label{c:cat=n}
Let $M^n$ be a closed manifold such that $\cat M=\dim M=n$. If $X$ is a closed manifold with $X\cong M$ then $\crit X=n+1=\crit M$.
\end{cor}

\begin{ex}
1. If $M=\R\P^n$ or $M=T^n$ and $X\cong M$ then $\crit M =\crit X$. This follows from \corref{c:cat=n} in view of examples \ref{e:torus}(1,2).

\m 2. We prove that $\crit X=n+1$ for $X\cong \cp^n, n\neq 2$. (In particular, $\crit \cp^n=n+1$ but we do not use it here, it appears a posteriori.) The case $n=1$ is trivial. For all $n$ we have
\[
\crit X \geq 1+\cat X=1+\cat \cp^n=n+1.
\]
 Finally, $H_i(X)=0=H^i(X)$ for $i$ odd. So, $r(\cp^n)=\{0,2,4,\ldots, 2n\}$ in \theoref{t:refin}. So, for $n\neq 2$ we have $\crit X\leq n+1$ by \theoref{t:refin}.

The case $n=2$ is an open question (although we know that $\crit \cp^2=3$).
\end{ex}

\m Keeping in mind \theoref{LS} on the relation between $\Crit$ and $\cat$, we pose the following question.

\begin{op}[Rudyak]
Given two manifolds $M,N$, assume that there exists a map $f: M \to N$ of degree 1. Is it true that $\cat M \geq \cat N$?
\end{op}

The question has affirmative answer for low-dimensional manifolds by Rudyak, \cite{Rud}.
\begin{thm}
For all maps $f: M^n \to N^n, n\leq 4$ of degree $\pm 1$, we have $\cat M \geq \cat N$.
\end{thm}

\begin{prop}
    If $M$ is a closed connected manifold with $\cat (M) =1$ then $M$ is a homotopy sphere.
\end{prop}

\begin{proof}
    Recall that $\pi_1(X)$ is free for all CW spaces (not necessarily manifolds) $X$ with $\cat (X)=1$ (i.e. the so-called co-$H$-spaces), see~ \cite[Prop. 2.4.3]{Ark}. For a simply connected closed manifold $M^n$ we claim that $H^i(M)=0$ for $i\neq 0,n$ (and hence $M$ is a homotopy sphere). Indeed, if $a\in H^k(M), a\neq 0$ for some $k\neq 0,n$ then there is a Poincar\'e dual class $b\in H^{n-k}(M)$ with $a\smile b\neq 0$. Hence the $\Z$-cup-length of $M$ is at least 2, and so $\cat (M)\geq 2$, which is a contradiction. 
 If $M$ is not simply connected then $\pi_1(M)$ is a non-trivial free group, and so $H_1(M)$ is a non-zero free abelian group, and so $H^1(M;\Z/2)\neq 0$. Hence, asserting as above, we see that $\Z/2$-cup-length of $M$ is at least 2. Thus, $\cat (M) \geq 2$, which is a contradiction.
\end{proof}

\begin{cor}[A weak version of Reeb's theorem]\label{reeb} If $M$ is a closed connected manifold with $\crit M=2$ then $M$ is a homotopy sphere.
\end{cor}

\begin{proof} Indeed, in this case $\cat M=1$. 
\end{proof}

\begin{cor}\label{homsphere}
   Let $f \colon M\to N$ be a map of degree $\pm1$. If $M$ is a homotopy sphere then $N$ is also a homotopy sphere. 
\end{cor}

\begin{proof}
    Let $\dim M=n$. We have $H_k(M)=0$ for $k\neq 0, n$ because $M$ is a homotopy sphere. Hence  $H_k(N)=0$ for $k\neq 0, n$ because of \propref{p:deg}. Furthermore, $\pi_1(N)=0$ by \propref{p:deg}.
\end{proof}

\section{Preliminares on 3-manifolds.}

A 3-manifold $M$ is {\em irreducible} if every embedded two-sphere $S^2\hookrightarrow M$
 bounds an embedded disk $D^3$ in $M$

\m A 3-manifold $M$ is {\em prime} if it cannot be expressed as a connected sum  $\displaystyle N_1\#N_2$ 
of two manifolds neither of which is the 3-sphere $S^3$.

\m Clearly, each irreducible $M^3$  is prime. A partial converse looks as follows.

\begin{lemma}\label{l:prim-not-irr}
If a 3-manifold $M$ is prime and not irreducible, then $M$ is a 2-sphere bundle over $S^1$. Furthermore, if $N$ is orientable then $M=S^1\times S^2$.
\end{lemma}

\begin{proof}
For the first claim see~\cite[Lemma 3.13]{H}. Hence, up to diffeomorphism, $M$ is either $S^1\times S^2$ or the mapping torus of the map
\[
h: S^2\to S^2, \quad h(x)=-x
\]
where $S^2$ is regarded as the set of unit vectors in $\R^3$. So, since the mapping torus is non-orientable, we conclude that $M=S^1\times S^2$.
\end{proof}

\begin{prop}\label{p:free}If M is a closed 3-manifold with non-trivial free fundamental
group, then M is not irreducible.
\end{prop}

\begin{proof}
See \cite[Corollary 4.6]{OR}
\end{proof}

\begin{lemma}\label{l:crit=3}
    Let $M$ be a closed, orientable 3-manifold. If $\cat(M)=2$ then $M$ is a connected sum of copies of $S^1\ts S^2$, $M=(\#_i(S^1\ts S^2))$, and thus $\crit M=3$.
\end{lemma} 

\begin{proof}
    By \cite[Theorem 3.15]{H} we have a decomposition 
\[
M=M_1\# \cdots \#M_n
\]
 where all manifolds $M_i$'s are primes. Note also that each manifold $M_i$ has non-trivial fundamental group by Perelman. Furthermore, since $\cat M=2$, each manifold $M_i$ has free and non-trivial fundamental group,  see \cite[p. 117, Corollary]{Sv} or \cite[Theorem 4.2]{GoGo}. 
Furthermore, each manifold $M_i$ is not irreducible by \propref{p:free}.
Hence, $M_i=S^1\ts S^2$ for all $i$ by \lemref{l:prim-not-irr}. 
\vskip 0.1cm
Summing up, if $\cat M=2$ then $M=(\#_i(S^1\ts S^2))$, and so $\crit M=3$ by \cite[Theorem 3.3]{Tak}. 
\end{proof}

 \section{3-manifolds and maps of degree 1.}\label{main}

The next theorem is the main result of this paper.
 \begin{thm}
   For all maps $f: M^n \to N^n, n\leq 3$ of degree $\pm 1$, we have $\crit M \geq \crit N$.  
 \end{thm}
 \begin{proof}
For closed manifolds, there is always a minimum and a maximum. So $\crit M, N\geq 2$. Also $\crit M, N \leq 4$ by \theoref{tak}. Now consider the cases of different $n$ separately.\\
     The only closed connected manifolds of dimension $1$ up to diffeomorphism is $S^1$. So, if $\dim M,N=1$ then $\crit M=\crit N=2$.\\
     Next, consider $n=2$. For $\crit N=2$, the inequality holds trivially. \\
By way of contradiction, assume that $\crit N=3$ and $\crit M=2$. We know that $M$ is a homotopy 2-sphere by \corref{reeb}. Hence, $N$ is also a homotopy 2-sphere by \corref{homsphere} and thus $N$ is the standard 2-sphere. That is a contradiction. So the result holds for $n=2$.\\
     Let $n=3$. So, by \theoref{tak}, $\crit M,N$ can be $2,3$ or $4$ only. The result holds for $\crit N=2$ trivially. Now we exclude the pair
\[
(\crit M=2,\,\crit N\geq 3).
\]
    Indeed, asserting as in case $\dim n=2$ we see that $N$ is a homotopy sphere. Using the well-known Perelman theorem on the Poincare conjecture, see e.g.~\cite{Mor} we conclude that $N$ is a standard 3-sphere, and hence $\crit N=2$. That is a contradiction. 

\m Now let $\crit N=4$ and prove that in his case $\Crit M=4$. By way of contradiction assume $\crit M=3$. By Lusternik--Schnirelman \theoref{LS} we conclude that $\cat M \leq 2$. So by ~\cite{Rud} we have $\cat N=1$ or $2$. If $\cat N=1$, then $N$ is a homotopy sphere by \corref{homsphere}. That is a contradiction. 

Finally, if $\cat N=2$ then by \lemref{l:crit=3} we have $\crit N=3$. This is a contradiction.
\end{proof}

\end{document}